\documentclass [a4paper,12pt]{article}
\usepackage{amsmath}
\usepackage{amsfonts}
\usepackage{indentfirst}
\usepackage{esint}
\usepackage{latexsym}
\usepackage{bbm}
\usepackage{amsfonts}
\usepackage{mathrsfs}
\usepackage{amssymb}
\usepackage{amsmath}
\usepackage{amsthm}
\usepackage{latexsym}
\usepackage{indentfirst}
\usepackage{enumitem}
\usepackage{bm}
\usepackage{esint}
\usepackage{cite}
\numberwithin{equation}{section}
\usepackage{color}
\allowdisplaybreaks
\usepackage{amsthm}
\usepackage{amssymb}
\usepackage{enumerate}
\usepackage{ulem}
\allowdisplaybreaks

\newtheorem{theorem}{Theorem}[section]
\newtheorem{lemma}[theorem]{Lemma}
\newtheorem{example}[theorem]{Example}
\newtheorem{thm}[theorem]{Theorem}
\newtheorem{prop}[theorem]{Proposition}
\newtheorem{cor}[theorem]{Corollary}

\newtheorem{rmk}[theorem]{Remark}

\makeatletter
\usepackage{amsmath}
\usepackage{amsfonts}
\usepackage{indentfirst}
\usepackage{esint}
\usepackage{latexsym}
\usepackage{color}
\UseRawInputEncoding
\usepackage{amsthm}
\usepackage{amssymb}
\usepackage{enumerate}
\usepackage{ulem}
\makeatletter

\newcommand{\Rmnum}[1]{\expandafter\@slowromancap\romannumeral #1@}
\makeatother
\begin{document}
\title{A Carleman-Type Inequality in Elliptic Periodic Homogenization}
\author{Yiping Zhang\footnote{Email:zhangyiping161@mails.ucas.ac.cn}\\\small{Academy of Mathematics and Systems Science, CAS;}\\
\small{University of Chinese Academy of Sciences;}\\
\small{Beijing 100190, P.R. China.}}
\date{}
\maketitle
\begin{abstract}
In this paper, for a family of second-order elliptic equations with rapidly oscillating periodic coefficients, we are interested in a Carleman-type inequality for these solutions satisfying an additional growth condition in elliptic periodic homogenization, which implies a three-ball inequality without an error term at a macroscopic scale. Moreover, if we replace the additional growth condition by the doubling condition at a macroscopic scale, then the three-ball inequality without an error term holds at any scale. The proof relies on the convergence of  $H^1$-norm for the solution and the compactness argument.
\end{abstract}
\section{Introduction}
Since T. Carleman's pioneer work \cite{MR0000334}, Carleman estimates have been indispensable tools for obtaining a three-ball (or three-cylinder) inequality and proving the unique continuation property for partial differential equations. In general, the Carleman estimates are weighted integral inequalities with suitable weight functions satisfying some convexity properties. The three-ball inequality is obtained by applying the Carleman estimates by choosing a suitable function. For Carleman estimates and the unique continuation properties for the elliptic and parabolic operators, we refer readers to  \cite{MR92067,MR140031,MR794370,MR1081811,escauriaza2000carleman,escauriaza2004unique,vessella2008quantitative,koch2009carleman} and their references therein for more results.

Over the last forty years, there is a vast and rich mathematical literature on homogenization. Most of these works are focused on qualitative results, such as proving the existence of a homogenized equation. However, until recently, nearly all of the quantitative theory, such as the convergence rates in $L^2$ and $H^1$, the $W^{1,p}$-estimates, the Lipschitz estimates and the asymptotic expansion of the Green functions and fundamental solutions, were confined in periodic homogenization. There are many good expositions on this topic, see for instance the books \cite{MR2839402,MR1329546,MR3838419}
for periodic case, see also the book \cite{MR3932093} for the stochastic case.

Recently, authors in \cite{armstrong2020largescale,kenig2021doubling,MR4194318,MR3952693} care about the propagation of smallness in homogenization theory, such as the approximate three-ball inequality in \cite{armstrong2020largescale,MR4194318} in elliptic periodic homogenization and the approximate two-sphere one-cylinder inequality in \cite{zhang2020approximate} in parabolic case, and the nodal sets and doubling conditions in \cite{kenig2021doubling,MR3952693} in elliptic homogenization, which are all related to the Carleman inequality in classical elliptic and parabolic theory and encourage us to deduce a Carleman-type inequality in elliptic periodic homogenization and left for further for the parabolic case.

 In this paper, we would like to deduce a Carleman-type inequality in elliptic periodic homogenization. According to the author's knowledge, this is the first attempt in homogenization theory. More precisely, We consider a family of second-order elliptic equations in divergence form with rapidly oscillating periodic coefficients,
\begin{equation}
\mathcal{L}_\varepsilon u_\varepsilon=:-\operatorname{div}\left(A(x/\varepsilon)\nabla u_\varepsilon\right)=0,
\end{equation}
where $1>\varepsilon>0$ and $A(y)=(a_{ij}(y))$ is a real symmetric $d\times d$ matrix-valued function in $\mathbb{R}^d$ for $d\geq 2$. Assume that $A(y)$ satisfies the following assumptions:

(i) Ellipticity: For some $0<\mu<1$ and all $y\in \mathbb{R}^d$, $\xi\in \mathbb{R}^d$, it holds that
\begin{equation}
\mu |\xi|^2\leq A(y)\xi\cdot\xi\leq \mu^{-1}|\xi|^2.
\end{equation}

(ii) Periodicity:
\begin{equation}A(y+z)=A(y) \quad \text{for } y\in \mathbb{R}^d \text{ and }z\in\mathbb{Z}^d.\end{equation}

(iii)Lipschitz continuity: There exist constants $\tilde{M}>0$ such that
\begin{equation}|A(x)-A(y)|\leq \tilde{M} |x-y|,\  \text{for any }x,y\in \mathbb{R}^d.\end{equation}

Let $B(x,r)=\left\{y\in \mathbb{R}^d:|y-x|<r\right\}$ and $B_r=B(0,r)$. For positive constants $M$, $N_1\geq 1$ and $N_2\geq 1$, let $u_\varepsilon\in H^2(B_3)$ be a solution of, and satisfies the following growth conditions,
\begin{equation}\int_{B_3}|u_\varepsilon|^2\leq M \max\left\{\left(\int_{B_2}|u_\varepsilon|^2\right)^{N_1},\left(\int_{B_2}|u_\varepsilon|^2\right)^{1/N_2}\right\},
\end{equation} then we could obtain the following result:
\begin{thm}
Assume that the coefficient matrix  $A$ is symmetric and satisfies the conditions $(1.2)$-$(1.4)$, and let $u_\varepsilon\in H^2(B_3)$ be a solution of $(1.1)$ and satisfies the growth condition $(1.5)$. Then, there exists $\varepsilon_0>0$, depending only on $d$, $\mu$, $\tilde{M}$, $M$, $N_1$ and $N_2$ such that for $0<\varepsilon\leq \varepsilon_0$, there holds the following Carleman-type inequality,
\begin{equation}\begin{aligned}
&\frac{C_0}{2} \int_{B_3}\left(\lambda^{4} \tau^{3} \varphi^{3} (u_\varepsilon\eta)^{2}+\lambda^{2} \tau \varphi\left|\nabla \left[(u_\varepsilon-\varepsilon\chi_j^\varepsilon\partial_j u_0)\eta\right]\right|^{2}\right) e^{2 \tau \varphi} d x\\
&\quad\quad \leq \int_{B_3}\left[\mathcal{L}_{\varepsilon} (u_\varepsilon\eta+\varepsilon \chi_j^\varepsilon\partial_j\eta u_\varepsilon)\right]^{2} e^{2 \tau \varphi} d x,
\end{aligned}\end{equation}
with $\varphi=e^{-\lambda|x|^2}$ and $u_0$ given by Theorem 2.2, for all $\lambda\geq \lambda_0$ and $\tau_0\leq\tau\leq 100\tau_0+ C(\lambda_0,\tau_0)\frac{||u_\varepsilon||_{L^2(B_3)}}{||u_\varepsilon||_{L^2(B_1)}}$ with $\lambda_0$ and $\tau_0$ defined in Proposition 2.6, depending only on $\mu$, where $C(\lambda_0,\tau_0)$ is a constant depending only on  $\lambda_0$ and $\tau_0$, and could be specified in the proof of Corollary 1.4 in Section 3. And $\eta\in C^\infty_0(B_{5/2}\setminus B_{1/2})$ is a fixed cutoff function such that
\begin{equation}0\leq \eta\leq 1,\ \eta=1 \text{ in }\overline{B_{7/3}\setminus B_{2/3}}.\end{equation}
\end{thm}Throughout this paper, we always assume that $\eta$ is a fixed cutoff function defined in Theorem 1.1.
We prove this  theorem by compactness argument. For the compactness method used in homogenization theory, we refer readers to \cite{MR910954,MR978702} for more details.
\begin{rmk}The growth condition $(1.5)$ allows that the function $u_\varepsilon$ grows at a speed of polynomials of any degree. For example,
if $u_\varepsilon$ behaves like $|x|^k$, for any $k\in \mathbb{N}^+$, then it is easy to see that the conditions holds with $M$ depending only on $\left\|u_\varepsilon/ |x|^k\right\|_{L^\infty(B_3)}$ and for any $N_1\geq \log 3/ \log2$.

Due to the rapid oscillation of the coefficient matrix $A(x/ \varepsilon)$, we could not expect a Carleman inequality in homogenization totally similar to the classical case. Moreover, it implies that the growth condition $(1.5)$ is necessary in compactness argument to obtain the Carleman-type inequality $(1.6)$ in Example 3.2. (Meanwhile, one may use other methods to derive a Carleman-type inequality in homogenization without the growth condition.)
\end{rmk}

\begin{rmk}The Carleman-type inequality $(1.6)$ continues to hold for the operator $\tilde{\mathcal{L}}_\varepsilon=-\operatorname{div}(A^\varepsilon \nabla)+\mathbf{b}^\varepsilon\nabla+c^\varepsilon+\lambda$ with $\mathbf{b}\in L^\infty$ and $c\in L^\infty$ being 1-periodic (the operator $\tilde{\mathcal{L}}_\varepsilon$ is positive by adding a large constant $\lambda$), since the $L^2$-norm as well as the $H^1$-norm convergence rates continue to hold for the solution $u_\varepsilon$ to the operator $\tilde{\mathcal{L}}_\varepsilon$ \cite{MR3466080}, and the Carleman inequality as well as the unique continuation property continue to hold for the homogenized operator $\tilde{\mathcal{L}}_0=-\operatorname{div}(\widehat{A} \nabla)+\mathcal{M}(\mathbf{b})\nabla+\mathcal{M}(c)+\lambda$ (Remark 2.7), where we have used the notations $F^\varepsilon=F(x/\varepsilon)$ for a function $F$ and $\mathcal{M}(G)=\int_Y G(y)dy$ for a 1-periodic function $G$.
\end{rmk}
The most trivial application of Theorem 1.1 is the following three-ball inequality at a macroscopic scale without an error term.

\begin{cor}
Assume that the coefficient matrix  $A$ is symmetric and satisfies the conditions $(1.2)$-$(1.4)$, and let $u_\varepsilon\in H^2(B_3)$ be a solution of $(1.1)$ and satisfy the growth condition $(1.5)$. Then, for some constant $C$, depending only on $d$, $\mu$, $\tilde{M}$, and $M$, there holds the following three-ball inequality without an error term,
\begin{equation}
||u_\varepsilon||_{L^2(B_2)}\leq C ||u_\varepsilon||^s_{L^2(B_1)}||u_\varepsilon||^{1-s}_{L^2(B_3)},
\end{equation} where $s=\frac{\alpha}{\alpha+\beta}$ with $\alpha=1-2e^{-4\lambda}$ and $\beta=2(e^{-4\lambda}-e^{-\frac{81}{16}\lambda})$ for any $\lambda\geq \lambda_0$ with $\lambda_0$ defined in Theorem 1.1.
\end{cor}

The Corollary above only implies the three-ball inequality at a macroscopic scale, in the following theorem, we could  obtain the three-ball inequality at every scale by using Corollary 1.4 and the uniform doubling conditions proved in \cite{MR3952693} in elliptic homogenization.

\begin{thm}Assume that $A$ is symmetric and satisfies the conditions $(1.2)$-$(1.4)$. Let $u_\varepsilon\in H^2(B_4)$ be a solution to the equation $\mathcal{L}_\varepsilon(u_\varepsilon)=0$ in $B_4$, and for some positive constant $M$, $u_\varepsilon$ satisfies the following doubling condition at a macroscopic scale,
\begin{equation*}
\fint_{B_4}|u_\varepsilon|^2\leq M \fint_{B_{2\sqrt{\mu}}}|u_\varepsilon|^2,
\end{equation*}then for any $0<r\leq 1/2$,
 there holds the following three-sphere inequality without an error term,
\begin{equation}
||u_\varepsilon||_{L^2(B_{2r})}\leq C ||u_\varepsilon||_{L^2(B_r)}^s||u_\varepsilon||_{L^2(B_{4r})}^{1-s},
\end{equation}with the same $s$ defined in Corollary 1.4 and $C$ depending only on $d$, $\mu$, $\tilde{M}$ and $M$.
\end{thm}
The first result about the approximate three-ball inequality was obtained by Kenig and Zhu in \cite{MR4194318} with the help of the asymptotic behavior of Green functions and the Lagrange interpolation technique, under the assumptions that the coefficient matrix $A$ is only H\"{o}lder continuous. Later on, an improvement of a sharp exponential error term (in the sense that if $A$ is only H\"{o}lder continuous, then the multiplicative factor must be  at least exponential) in the error bound for the approximate three-ball inequality (under certain extra conditions) was discovered by Armstrong, Kuusi, and Smart in \cite{armstrong2020largescale}, as a consequence of the large-scale “analyticity”. Meanwhile, an approximate two-sphere one-cylinder inequality in parabolic periodic homogenization was obtained by the first author in \cite{zhang2020approximate}, under the assumptions tha the coefficient matrix $A(x,t)$ is only H\"{o}lder continuous, with the help of the asymptotic behavior of fundamental solutions and the Lagrange interpolation technique.

Recently, the three-ball inequality without an error term was discovered by Kenig, Zhu and Zhuge in \cite{kenig2021doubling} under the assumptions that $A\in C^{0,1}$ and $u_\varepsilon$ satisfies a doubling condition at a macroscopic scale, with the help of the approximate three-ball inequality with a sharp exponential error term obtained in \cite{armstrong2020largescale}. At this stage, we should compare the three-ball inequality obtained in \cite{kenig2021doubling} with the result proved in Theorem 1.5 in this paper. The result  reads that:

\textbf{Theorem}. Assume that the coefficient matrix $A$ satisfies $(1.2)$-$(1.4)$. For every $\tau>0$, there exists $C>1$ and $\theta\in (0,1/2)$ depending only on $d$, $\mu$ and $\lambda$ such that if $u_\varepsilon$ is a weak solution of $\mathcal{L}_\varepsilon(u_\varepsilon)=0$ in $B_1$ satisfying
$$\int_{B_1}|u_\varepsilon|^2dx\leq M \int_{B_\theta}|u_\varepsilon|^2dx,$$
then for any $r\in(0,1)$,
$$\int_{B_\theta}|u_\varepsilon|^2dx\leq \exp(\exp(CM^\tau)) \int_{B_{\theta r}}|u_\varepsilon|^2dx$$
and
$$\int_{B_{\theta r}}|u_\varepsilon|^2dx\leq \exp(\exp(CM^\tau)) \left(\int_{B_{\theta^2 r}}|u_\varepsilon|^2dx\right)^{\tau_1}\left(\int_{B_{ r}}|u_\varepsilon|^2dx\right)^{1-\tau_1}$$
for any $0<\tau_1<1$. It is clear that the authors in \cite{kenig2021doubling} have found an explicit estimate for the constant $C(M)$ in the doubling condition and in the three-ball inequality, with an unknown $\theta$. However, in our Theorem 1.5, we could state $\theta$ explicitly and obtain the three-ball inequality more directly with an unknown constant $C(M)$.
Throughout this paper, with $Y=[0,1)^d\cong \mathbb{R}^d/\mathbb{Z}^d$, we use the following notation

$$H^m_{\text{per}}(Y)=:\left\{f\in H^m(Y) \text{ and }f\text{ is 1-periodic with }\fint_Yfdy=0\right\}, $$
and we will write $\partial_{x_i}$ as $\partial_i$, $F^\varepsilon=F(x/\varepsilon)$ for a function $F$ and $\mathcal{M}(G)=\int_Y G(y)dy$ for a 1-periodic function $G$ if the context is understand.
\section{Preliminaries}
Assume that $A=A(y)$ satisfies the conditions $(1.2)$-$(1.3)$. Let $\chi(y)\in H^1_{\text{per}}(Y;\mathbb{R}^d)$ denote the first order corrector for $\mathcal{L}_\varepsilon$, where $\chi_j$ for $j=1,\cdots,d$ is the unique 1-periodic function in $H^1_{\text{per}}(Y)$ such that
\begin{equation}\begin{cases}
\mathcal{L}_1(\chi_j)=-\mathcal{L}_1(y_j)\quad \text{in }Y,\\
\fint_Y \chi_jdy=0.
\end{cases}\end{equation}
By the classical Schauder estimates, $\chi\in C^{1,\alpha}$ if $A\in C^{0,\alpha}$. The homogenized operator for $\mathcal{L}_\varepsilon$ is given by $\mathcal{L}_0=-\operatorname{div}(\widehat{A}\nabla)$, where $\widehat{A}=(\widehat{a}_{ij})_{d\times d}$ and
\begin{equation}\widehat{a}_{i j}=\fint_{Y}\left[a_{i j}+a_{i k} \frac{\partial}{\partial y_{k}}\left(\chi_{j}\right)\right](y) d y.\end{equation}
It is well-known that the homogenized matrix $\widehat{A}$ also satisfies the ellipticity condition $(1.2)$ with the same $\mu$. What's more, if $A$ is symmetric, the same is also true for $\widehat{A}$. We refer the readers to \cite{MR3838419} for the proofs.

Denote the so-called flux correctors $b_{ij}$ by
\begin{equation}
b_{ij}(y)=\widehat{a}_{ij}-a_{ij}(y)-a_{ik}(y)\frac{\partial\chi_j(y)}{ \partial{y_k}},
\end{equation}
where $1\leq i,j\leq d$.
\begin{lemma}
Suppose that $A$ satisfies the conditions $(1.2)$ and $(1.3)$. For $1\leq i,j,k\leq d$, there exists $F_{ijk}\in H^1_{\text{per}}(Y)\cap L^\infty(Y)$ such that
\begin{equation}
b_{ij}=\frac{\partial}{ \partial {y_k}}F_{kij}\ \text{ and }\ F_{kij}=-F_{ikj}.
\end{equation}
\begin{proof}See \cite[Remark 2.1]{kenig2014periodic}.\end{proof}
\end{lemma}

The following theorem states the existence of $u_0$ in Theorem 1.1 and  is used to control the second term on the left hand side of $(1.6)$.
\begin{thm}
Suppose that $A$ is symmetric and satisfies the conditions $(1.2)$ and $(1.3)$. Let $u_\varepsilon\in H^1(B_3)$ be the weak solution of equation
$\mathcal{L}_\varepsilon(u_\varepsilon)=0 \text{ in }B_3.$ Then there exists $u_0\in H^1(B_{11/4})$ such that $\mathcal{L}_0(u_0)=0 \text{ in }B_{11/4}$, and
\begin{equation}
||u_\varepsilon-u_0||_{L^2(B_{11/4})}\leq C \sqrt{\varepsilon }||u_\varepsilon||_{L^2(B_3)},
\end{equation} where $C$ depends only on $d$ and $\mu$.
\end{thm}
\begin{proof}
Due to the Caccioppoli's inequality and the co-area formula, there exists $r_0\in [\frac{11}{4},\frac{23}{8}]$ such that
\begin{equation}
\int_{\partial B_{r_0}}|u_\varepsilon|^2dS+\int_{\partial B_{r_0}}| \nabla u_\varepsilon|^2dS\leq C\int_{B_3}|u_\varepsilon|^2dx.
\end{equation}Then, we could consider the following Dirichlet problem,
\begin{equation}\begin{cases}
\mathcal{L}_\varepsilon(u_\varepsilon)=0 \text{ in }B_{r_0}\\
u_\varepsilon\in H^1(\partial B_{r_0})\text{ with }||u_\varepsilon||_{H^1(\partial B_{r_0})}\leq C||u_\varepsilon||_{L^2(B_3)}.
\end{cases}\end{equation}
And let $u_0$ satisfies the following equation,
\begin{equation}\begin{cases}
\mathcal{L}_0(u_0)=0 \text{ in }B_{r_0}\\
u_0=u_\varepsilon \text{ on }\partial B_{r_0}.
\end{cases}\end{equation}
Since $A$ is symmetric, therefore, it follows from the homogenization theory that there holds (for the proof, see \cite{MR3838419} for example)
\begin{equation}\begin{aligned}
&||u_\varepsilon-u_0||_{L^2(B_{11/4})}\leq ||u_\varepsilon-u_0||_{L^2(B_{r_0})}\\
\leq &C \varepsilon || \nabla u_0||_{L^2(B_{r_0})}+C \sqrt{\varepsilon}||u_\varepsilon||_{H^1(\partial B_{r_0})}\\
\leq &C \sqrt{\varepsilon}||u_\varepsilon||_{H^1(\partial B_{r_0})}\leq C\varepsilon ||u_\varepsilon||_{L^2(B_3)},
\end{aligned}\end{equation}
where we have used the $H^1$ estimate for $u_0$ and $(2.6)$ in the third line in inequality $(2.9)$. Thus we have completed this proof.
\end{proof}
\begin{rmk}
In Theorem 2.2, if we additionally assume that $A$ is Lipschitz continuous, then there exists $u_0\in H^1(B_{11/4})$ such that $\mathcal{L}_0(u_0)=0 \text{ in }B_{11/4}$, and
\begin{equation*}
||u_\varepsilon-u_0||_{L^\infty(B_{5/2})}\leq C \varepsilon ||u_\varepsilon||_{L^2(B_3)}.
\end{equation*}The main ideal of this proof is due to Lin and Shen \cite{MR3952693}, and we omit it here.
\end{rmk}
Next, we introduce the following well-known Div-Curl lemma whose proof may be found in \cite{MR3838419}.
\begin{lemma}
Let $\{u_k\}$ and $\{v_k\}$ be two bounded sequences in $L^2(\Omega;\mathbb{R}^d)$ with $\Omega$ being a bounded Lipschitz domain.
Suppose that\\
$(i)$ $u_k\rightharpoonup u$ and $v_k\rightharpoonup v$ weakly in $L^2(\Omega;\mathbb{R}^d)$;\\
$(ii)$ $\operatorname{curl}(u_k)=0$ in $\Omega$ and $\operatorname{div}(v_k)=f$ strongly in $H^{-1}(\Omega)$.\\
Then there holds
\begin{equation*}
\int_\Omega(u_k\cdot v_k)\varphi dx\rightarrow \int_\Omega (u\cdot v)\varphi dx
\end{equation*}as $k\rightarrow \infty$, for any scalar function $\varphi\in C^1_0(\Omega)$.
\end{lemma}
The following interior Caccioppoli's inequality with weights will be used in the proof of Theorem 1.1.
\begin{lemma}(interior Caccioppoli's inequality with weights)
Assume that $A$ satisfies the condition $(1.2)$, and $u_\varepsilon\in H^1(B_3)$ is a weak solution of $\mathcal{L}_\varepsilon(u_\varepsilon)=0$ in $B_3$. Let $0\leq s_1<s_2<s_3<s_4\leq 3$,  then there holds
\begin{equation}\begin{aligned}
\int_{B_{s_3}\backslash B_{s_2}}|\nabla u_\varepsilon|^2 e^{2\tau\varphi}dx \leq C&\left(\frac{1}{s_4-s_3}+\frac{1}{s_2-s_1}\right)^2\int_{B_{s_4}\backslash B_{s_1}}|u_\varepsilon|^2e^{2\tau\varphi}dx\\
&+C\lambda^2\tau^2\int_{B_{s_4}\backslash B_{s_1}}|x|^2|u_\varepsilon|^2\varphi^2e^{2\tau\varphi}dx,
\end{aligned}\end{equation}
where $C$ depends only on $d$ and $\mu$ and $\varphi=e^{-\lambda|x|^2}$ with $\lambda$ and $\tau$ being positive constants.
\end{lemma}
\begin{proof}
The proof is standard. Choose a cutoff function $0\leq\rho(x)\leq 1$, such that $\rho(x)=1$ if $x\in B_{s_3}\setminus B_{s_2}$ and $\rho(x)=0$ if $x\notin B_{s_4}\setminus B_{s_1}$ with $|\nabla \rho|\leq C\left(\frac{1}{s_3-s_4}+\frac{1}{s_1-s_2}\right)$. Then testing the equation $\mathcal{L}_\varepsilon(u_\varepsilon)=0$ in $B_3$ with $u_\varepsilon e^{2\tau \varphi}\rho^2$ yields that
\begin{equation}
\int_{B_3}A^\varepsilon \nabla u_\varepsilon \nabla u_\varepsilon e^{2\tau\varphi}\rho^2-2\lambda\tau \int_{B_3}A^\varepsilon\nabla u_\varepsilon\cdot x  u_\varepsilon \varphi e^{2\tau\varphi}\rho^2+2\int_{B_3}A^\varepsilon\nabla u_\varepsilon\nabla \rho u_\varepsilon \rho e^{2\tau\varphi}=0.
\end{equation}
Then, it follows from the Cauchy inequality that
\begin{equation}
\int_{B_3}|\nabla u_\varepsilon|^2\rho^2e^{2\tau\varphi}\leq C\lambda^2\tau^2 \int_{B_3}|x|^2|u_\varepsilon|^2\varphi^2\rho^2 e^{2\tau\varphi}+C \int_{B_3}|u_\varepsilon|^2|\nabla \rho|^2e^{2\tau\varphi}.
\end{equation}
Thus, we have completed this proof after noting the choice of $\rho$.
\end{proof}
At the end of this section, we introduce the following Carleman inequality for the homogenized operator $\mathcal{L}_0=-\operatorname{div}(\widehat{A}\nabla)$, whose proof may be found in \cite{MR3495389}.

\begin{prop}(Carleman inequality) Assume that $A$ is symmetric and satisfies the conditions $(1.2)$-$(1.4)$, then there exist three positive constants $C_0$ , $\lambda_0$ and $\tau_0$ that can depend only on $\mu$, such that
\begin{equation}
C_0 \int_{B_3}\left(\lambda^{4} \tau^{3} \varphi^{3} (v\tilde{\eta})^{2}+\lambda^{2} \tau \varphi|\nabla (v\tilde{\eta})|^{2}\right) e^{2 \tau \varphi} d x \leq \int_{B_3}\left[\mathcal{L}_{0} (v\tilde{\eta})\right]^{2} e^{2 \tau \varphi} d x,
\end{equation}
with $\varphi=e^{-\lambda|x|^2}$, for all $v\in H^2(B_3)$, $\tilde{\eta}\in C^\infty_0(B_3\setminus B_{1/2})$, $\lambda\geq \lambda_0$ and $\tau\geq\tau_0$.
\end{prop}

\begin{rmk}
The Carleman inequality $(2.13)$ continues to hold for the operator $\tilde{\mathcal{L}}=-\operatorname{div}(\tilde{A}\nabla)+\mathbf{B}\cdot \nabla +c$ with symmetric $\tilde{A}$ satisfying the ellipticity condition and being Lipschitz continuous, $\mathbf{B}\in L^\infty(B_3)^d$ and $c\in L^\infty(\Omega)$, where the constants $C_0$, $\lambda_0$ and $\tau_0$ depends only on $\mu$ and the $L^\infty(B_3)$-norm of $\mathbf{B}$ and $c$.
\end{rmk}

\section{Carleman inequality}
To proceed further, we first need to calculate the term $\mathcal{L}_\varepsilon(u_\varepsilon \eta+\varepsilon\chi_j^\varepsilon\partial_j\eta u_\varepsilon)$ on the right hand side of $(1.6)$, which will be stated in the following lemma.
\begin{lemma}
Suppose that $A$ is symmetric and satisfies the conditions $(1.2)$-$(1.4)$, and let $u_\varepsilon\in H^2(B_3)$ be a solution to the equation $\mathcal{L}_\varepsilon(u_\varepsilon)=0$ in $B_3$, then there holds
\begin{equation}\begin{aligned}
&-\mathcal{L}_\varepsilon(u_\varepsilon \eta+\varepsilon\chi_j^\varepsilon\partial_j\eta u_\varepsilon)\\
=&2A^\varepsilon \nabla u_\varepsilon \nabla\eta+A^\varepsilon \nabla^2\eta u_\varepsilon+A^\varepsilon\nabla_y\chi_j^\varepsilon\nabla\partial_j\eta u_\varepsilon
+2A^\varepsilon\nabla u_\varepsilon\nabla_y\chi_j^\varepsilon\partial_j\eta\\
&+\operatorname{div}_y(A^\varepsilon\chi_j^\varepsilon)\nabla\partial_j\eta u_\varepsilon
+\varepsilon A^\varepsilon\chi_j^\varepsilon\nabla^2\partial_j\eta u_\varepsilon
+2\varepsilon A^\varepsilon\chi_j^\varepsilon\nabla\partial_j\eta \nabla u_\varepsilon\quad\text{in }L^2(B_3).
\end{aligned}\end{equation}
\end{lemma}
\begin{proof} Since $u_\varepsilon\in H^2(B_3)$ satisfies $\mathcal{L}_\varepsilon(u_\varepsilon)=0$ in $B_3$, then it is easy to see that
\begin{equation}\begin{aligned}
-\mathcal{L}_\varepsilon(u_\varepsilon \eta)&=\operatorname{div}(A^\varepsilon \nabla u_\varepsilon\eta+A^\varepsilon\nabla \eta u_\varepsilon)\\
&=2A^\varepsilon \nabla u_\varepsilon \nabla\eta+\frac{1}{\varepsilon}\partial_{y_i}a_{ij}^\varepsilon \partial_j \eta u_\varepsilon+A^\varepsilon \nabla^2\eta u_\varepsilon.
\end{aligned}\end{equation}
In order to cancel out the term $\frac{1}{\varepsilon}\partial_{y_i}a_{ij}^\varepsilon \partial_j \eta u_\varepsilon$, we need to consider the term $\varepsilon\chi_j^\varepsilon\partial_j\eta u_\varepsilon$. Then, in view of the definition of the first order corrector $\chi_j$ in $(2.1)$, we come to the following equality,
\begin{equation}\begin{aligned}
-\mathcal{L}_\varepsilon(\varepsilon\chi_j^\varepsilon\partial_j\eta u_\varepsilon)
=&\operatorname{div}(A^\varepsilon\nabla_y\chi_j^\varepsilon\partial_j\eta u_\varepsilon+\varepsilon A^\varepsilon\chi_j^\varepsilon\nabla\partial_j\eta u_\varepsilon+\varepsilon A^\varepsilon\chi_j^\varepsilon\partial_j\eta \nabla u_\varepsilon)\\
=&-\frac{1}{\varepsilon}\partial_{y_i}a_{ij}^\varepsilon \partial_j \eta u_\varepsilon
+A^\varepsilon\nabla_y\chi_j^\varepsilon\nabla\partial_j\eta u_\varepsilon
+2A^\varepsilon\nabla_y\chi_j^\varepsilon\partial_j\eta\nabla u_\varepsilon\\
&+\operatorname{div}_y(A^\varepsilon\chi_j^\varepsilon)\nabla\partial_j\eta u_\varepsilon
+\varepsilon A^\varepsilon\chi_j^\varepsilon\nabla^2\partial_j\eta u_\varepsilon
+2\varepsilon A^\varepsilon\chi_j^\varepsilon\nabla\partial_j\eta \nabla u_\varepsilon,
\end{aligned}\end{equation}
where we have used the following equality
$$\operatorname{div}(\varepsilon A^\varepsilon\chi_j^\varepsilon\partial_j\eta \nabla u_\varepsilon)=\varepsilon A^\varepsilon\chi_j^\varepsilon\nabla\partial_j\eta \nabla u_\varepsilon
+A^\varepsilon\nabla u_\varepsilon\nabla_y\chi_j^\varepsilon\partial_j\eta,$$

in the above equation. Consequently, we have
\begin{equation}\begin{aligned}
&-\mathcal{L}_\varepsilon(u_\varepsilon \eta+\varepsilon\chi_j^\varepsilon\partial_j\eta u_\varepsilon)\\
=&2A^\varepsilon \nabla u_\varepsilon \nabla\eta+A^\varepsilon \nabla^2\eta u_\varepsilon+A^\varepsilon\nabla_y\chi_j^\varepsilon\nabla\partial_j\eta u_\varepsilon
+2A^\varepsilon\nabla u_\varepsilon\nabla_y\chi_j^\varepsilon\partial_j\eta\\
&+\operatorname{div}_y(A^\varepsilon\chi_j^\varepsilon)\nabla\partial_j\eta u_\varepsilon
+\varepsilon A^\varepsilon\chi_j^\varepsilon\nabla^2\partial_j\eta u_\varepsilon
+2\varepsilon A^\varepsilon\chi_j^\varepsilon\nabla\partial_j\eta \nabla u_\varepsilon,
\end{aligned}\end{equation}which completes the proof of Lemma 3.1.
\end{proof}
Now we are ready to give the proof of Theorem 1.1.\\
\textbf{Proof of Theorem 1.1}. We prove the result by contraction. Suppose that there exist sequence $\{\varepsilon_k\}\subset
\mathbb{R}^+$, $\{A_k\}$ being symmetric and satisfying $(1.2)$-$(1.4)$, $\{u_k\}\subset H^2(B_3)$, $\{u_{k,0}\}\in H^1(B_{11/4})$ given by Theorem 2.2, $\{\lambda_k\}$ satisfying $\lambda_k\geq \lambda_0$ and $\{\tau_k\}$ satisfying
$\tau_0\leq\tau_k\leq C(\lambda_0,\tau_0)\frac{||u_k||_{L^2(B_3)}}{||u_k||_{L^2(B_1)}}+100\tau_0$, such that $\varepsilon_k\rightarrow 0$, and
\begin{equation}\operatorname{div}(A_k(x/\varepsilon_k)\nabla u_k)=0\text{ in }B_3,\end{equation}
\begin{equation}\int_{B_3}|u_k|^2\leq M \max\left\{\left(\int_{B_2}|u_k|^2\right)^{N_1},\left(\int_{B_2}|u_k|^2\right)^{1/N_2}\right\}
\end{equation}
\begin{equation}\operatorname{div}(\widehat{A_k}\nabla u_{k,0})=0\text{ in }B_{11/4},\end{equation}
with
\begin{equation}||u_k-u_{k,0}||_{L^2(B_{11/4})}\leq C \sqrt{\varepsilon_k} ||u_k||_{L^2(B_3)},\end{equation}
and
\begin{equation}\begin{aligned}
&\frac{C_0}{2} \int_{B_3}\left(\lambda_k^{4} \tau_k^{3} \varphi_k^{3} (u_k\eta)^{2}+\lambda_k^{2} \tau_k \varphi_k\left|\nabla \left[(u_k-\varepsilon_k\chi_{k,j}^{\varepsilon_k}\partial_j u_{k,0})\eta\right]\right|^{2}\right) e^{2 \tau_k \varphi_k} d x\\
&\quad\quad > \int_{B_3}\left[\operatorname{div}\left(A_k^{\varepsilon_k}\nabla (u_k\eta+\varepsilon_k \chi_{k.j}^{\varepsilon_k}\partial_j\eta u_k)\right)\right]^{2} e^{2 \tau_k \varphi_k} d x,
\end{aligned}\end{equation}
where $\chi_{k,j}$ is the $j$-th corrector defined in $(2.1)$ for the operator $A_k$.
Since $\widehat{A_k}$ is symmetric and bounded in $\mathbb{R}^{d\times d}$, we may assume that
\begin{equation}
\widehat{A_k}\rightarrow H\end{equation} for some symmetric matrix $H$ satisfying the ellipticity condition $(1.2)$. Due to $\lambda_k^4\varphi_k^3=\lambda_k^4e^{-3\lambda_k|x|^2}\rightarrow 0$ and $\lambda_k^2\varphi_k=\lambda_k^2e^{-\lambda_k|x|^2}\rightarrow 0$ as $\lambda_k\rightarrow \infty$ if $|x|>1/2$, then  we could assume that
\begin{equation}
\limsup_{k\rightarrow\infty}\lambda_k=\lambda_\infty<+\infty.
\end{equation}
 By multiplying a constant to $u_k$, we may assume that
\begin{equation}||u_k||_{L^2(B_3)}=1.\end{equation}
By Caccioppoli's inequality, this implies that $\{u_k\}$ is bounded in $H^1(B_r)$ for any $0<r<3$. Therefore, by passing to a subsequence  still denoted by $\{u_k\}$, we may further assume that
\begin{equation}u_k\rightharpoonup u \text{ weakly in } H^1(B_r),\end{equation}
\begin{equation}u_k\rightarrow u \text{ strongly in } L^2(B_r),\end{equation}
\begin{equation}A_k(x/\varepsilon_k)\nabla u_k\rightharpoonup F\text{ weakly in } L^2(B_r)\end{equation}
for any $0<r<3$, where $u\in H^1_{\text{loc}}(B_3)$ and $F\in L^2_{\text{loc}}(B_3)$. It follows from the theory of homogenization (see e.g. \cite{MR3838419}) that $F=H\nabla u$ and
\begin{equation}
\operatorname{div}(H\nabla u)=0 \text{ in } B_3\text{ with } ||u||_{L^2(B_3)}\leq1.
\end{equation}
If we write $\chi_{k,j}(y)$ with $j=1,\cdots,d$ as the first order corrector for $-\operatorname{div}(A_k(x/\varepsilon_k)\nabla)$, then it is easy to see that
\begin{equation}\begin{aligned}
&\operatorname{div}[A_k(x/\varepsilon_k)\nabla(u_k-u-\varepsilon_k\chi_{k,j}(x/\varepsilon_k)\partial_ju)]\\
=&\operatorname{div}(H\nabla u)-\operatorname{div}(A_k(x/\varepsilon_k)\nabla u)
-\operatorname{div}(A_k(x/\varepsilon_k)\nabla\chi_{k,j}(x/\varepsilon_k)\partial_ju)\\
&-\varepsilon_k\operatorname{div}(A_k(x/\varepsilon_k)\chi_{k,j}(x/\varepsilon_k)\nabla\partial_ju)\\
=&\operatorname{div}((H-\widehat{A_k})\nabla u)
-\operatorname{div}((A_k(x/\varepsilon_k)-\widehat{A_k}+A_k(x/\varepsilon_k)\nabla\chi_{k}(x/\varepsilon_k))\nabla u)\\
&-\varepsilon_k\operatorname{div}(A_k(x/\varepsilon_k)\chi_{k,j}(x/\varepsilon_k)\nabla\partial_ju)\\
=&\operatorname{div}((H-\widehat{A_k})\nabla u)
-\varepsilon_k\partial_{x_l}\left\{F_{k,lij}\partial^2_{ij}u\right\}
-\varepsilon_k\operatorname{div}(A_k(x/\varepsilon_k)\chi_{k,j}(x/\varepsilon_k)\nabla\partial_ju)\text{ in }B_3,
\end{aligned}\end{equation}
where $F_{k,lij}=-F_{k,ilj}\in H^1_{\text{per}}(Y)\cap L^\infty(Y)$ given by Lemma 2.1 after replacing the coefficient matrix $A$ by $A_k$ in this lemma.

Consequently, it follows from the interior $W^{1,p}$ estimate \cite{MR2463969}, $(3.10)$ and $(3.14)$ that
\begin{equation}
||\nabla(u_k-u-\varepsilon_k\chi_{k,j}(x/\varepsilon_k)\partial_ju)||_{L^p(B_r)}\rightarrow 0, \text{ for any }1<p<\infty \text{ and }r<3.
\end{equation}
Next, note that
\begin{equation}
||u_{k,0}-u||_{L^2(B_{11/4})}\leq ||u_{k,0}-u_k||_{L^2(B_{11/4})}+||u_k-u||_{L^2(B_{11/4})}\rightarrow 0, \text{ as }k\rightarrow\infty,
\end{equation} and $u_{k,0}-u$ satisfies
\begin{equation}
\operatorname{div}(\widehat{A_k}\nabla(u_{k,0}-u))=(H-\widehat{A_k})\nabla^2 u\quad \text{ in }B_{11/4}.
\end{equation}
Then, it follows from the interior $H^k$ estimates for harmonic functions and $(3.19)$ that
\begin{equation}||u_{k,0}-u||_{H^k(B_r)}\rightarrow 0, \ \forall k\in \mathbf{N}^+ \text{ and }r<11/4.\end{equation}
Therefore, there holds
\begin{equation}\begin{aligned}
||\nabla(u_k-u-\varepsilon_k\chi_{k,j}(x/\varepsilon_k)\partial_ju_{k,0})||_{L^p(B_r)}\rightarrow 0 \text{ for any }r<11/4 \text { and }1<p<\infty.
\end{aligned}\end{equation}
In view of Lemma 3.1, we have
\begin{equation}\begin{aligned}
&\operatorname{div}\left(A_k(x/\varepsilon_k)\nabla(u_k \eta+\varepsilon\chi_{k,j}(x/\varepsilon_k)\partial_j\eta u_k)\right)\\
=&2A_k^{\varepsilon_k} \nabla u_k \nabla\eta+A_k^{\varepsilon_k} \nabla^2\eta u_k+A^{\varepsilon_k}_k\nabla_y\chi_{k,j}^{\varepsilon_k}\nabla\partial_j\eta u_k
+2A_k^{\varepsilon_k}\nabla u_k\nabla_y\chi_{k,j}^{\varepsilon_k}\partial_j\eta\\
&+\operatorname{div}_y(A_k^{\varepsilon_k}\chi_{k,j}^{\varepsilon_k})\nabla\partial_j\eta u_k
+{\varepsilon_k} A_k^{\varepsilon_k}\chi_{k,j}^{\varepsilon_k}\nabla^2\partial_j\eta u_k
+2{\varepsilon_k} A_k^{\varepsilon_k}\chi_{k,j}^{\varepsilon_k}\nabla\partial_j\eta \nabla u_k.
\end{aligned}\end{equation}
It is easy to see that
\begin{equation}\begin{aligned}
&2A_k^{\varepsilon_k} \nabla u_k \nabla\eta+A_k^{\varepsilon_k} \nabla^2\eta u_k+A^{\varepsilon_k}_k\nabla_y\chi_{k,j}^{\varepsilon_k}\nabla\partial_j\eta u_k\\
&\quad\rightharpoonup 2H\nabla u\nabla\eta+H\nabla^2\eta u \text{ weakly in }L^2(B_3).
\end{aligned}\end{equation}
Since $\int_Y \nabla_y\chi_{k,j}(y)dy=0$ and $||A_k^{\varepsilon_k}\nabla u_k\nabla_y\chi_{k,j}^{\varepsilon_k}\partial_j\eta||_{L^2(B_3)}\leq C$ with $C$ independent of $k$, then it follows from the so-called div-curve Lemma (Lemma 2.4) that
\begin{equation}
A_k^{\varepsilon_k}\nabla u_k\nabla_y\chi_{k,j}^{\varepsilon_k}\partial_j\eta\rightharpoonup 0\text{ weakly in }L^2(B_3).
\end{equation}
Meanwhile, we could easily obtain the following weak convergence,
\begin{equation}\begin{aligned}
&\operatorname{div}_y(A_k^{\varepsilon_k}\chi_{k,j}^{\varepsilon_k})\nabla\partial_j\eta u_k
+{\varepsilon_k} A_k^{\varepsilon_k}\chi_{k,j}^{\varepsilon_k}\nabla^2\partial_j\eta u_k
+2{\varepsilon_k} A_k^{\varepsilon_k}\chi_{k,j}^{\varepsilon_k}\nabla\partial_j\eta \nabla u_k\\
&\quad\quad\rightharpoonup 0\text{ weakly in }L^2(B_3).
\end{aligned}\end{equation}
Consequently, combining $(3.23)$-$(3.26)$ yields that
\begin{equation}\begin{aligned}
&\operatorname{div}\left(A_k(x/\varepsilon_k)\nabla(u_k \eta+\varepsilon\chi_{k,j}(x/\varepsilon_k)\partial_j\eta u_k)\right)\\
&\rightharpoonup 2H\nabla u\nabla\eta+H\nabla^2\eta u= \operatorname{div}(H\nabla(u\eta))\text{ weakly in }L^2(B_3).
\end{aligned}\end{equation}
To proceed, we first consider that, there exists some constant $\tau_\infty>0$, such that
\begin{equation}
\limsup_{k\rightarrow\infty}\tau_k=\tau_\infty<+\infty.
\end{equation}
Then, letting $k\rightarrow \infty$ along $(3.11)$ and $(3.28)$ on the both sides of $(3.9)$ yields that
\begin{equation}\begin{aligned}
&\frac{C_0}{2} \int_{B_3}\left(\lambda_\infty^{4} \tau_\infty^{3} \varphi_\infty^{3} (u\eta)^{2}+\lambda_\infty^{2} \tau_\infty \varphi_\infty\left|\nabla \left(u\eta\right)\right|^{2}\right) e^{2 \tau_\infty \varphi_\infty} d x\\
& \geq \int_{B_3}\left[\operatorname{div}(H\nabla(u\eta))\right]^{2} e^{2 \tau_\infty \varphi_\infty} d x\\
&\geq C_0 \int_{B_3}\left(\lambda_\infty^{4} \tau_\infty^{3} \varphi_\infty^{3} (u\eta)^{2}+\lambda_\infty^{2} \tau_\infty \varphi_\infty\left|\nabla \left(u\eta\right)\right|^{2}\right) e^{2 \tau_\infty \varphi_\infty} d x,
\end{aligned}\end{equation} where we have used the Carleman inequality (Proposition 2.6) for the matrix coefficient $H$ in the last inequality in $(3.29)$ and $\varphi_\infty=e^{-\lambda_\infty |x|^2}$.
It follows from $(3.29)$ and the unique continuation for harmonic function that
\begin{equation}
u\equiv 0 \text { in }B_3,
\end{equation} which contradicts to the conditions $(3.6)$, $(3.12)$ and $(3.14)$.

Next, we consider the case \begin{equation}
\limsup_{k\rightarrow\infty}\tau_k\rightarrow+\infty.
\end{equation}In view of $\tau_0\leq\tau_k\leq C(\lambda_0,\tau_0)\frac{||u_k||_{L^2(B_3)}}{||u_k||_{L^2(B_1)}}+100\tau_0$, then it follows from $(3.12)$ that
\begin{equation*}
||u_k||_{L^2(B_1)}\rightarrow0.
\end{equation*}Then, we could obtain
\begin{equation*}
u\equiv 0 \text{ in }B_1,
\end{equation*}which implies that
\begin{equation*}
u\equiv 0 \text { in }B_3,
\end{equation*} due to the unique continuation for harmonic function. Thus leads to a contraction again.\qed

The growth condition $(3.6)$ plays an important role in compactness argument. In the following example, we could construct a counterexample without the growth condition $(3.6)$.
\begin{example}If we consider $A_k=\Delta$ for any $k\geq 1$, then there exists a sequence of harmonic functions $\{u_k\}$ such that
$$\Delta u_k=0\text{ in }B_4,\ \int_{B_3}|u_k|^2dx=1\text{ and }\int_{B_1}|u_k|^2dx\rightarrow 0 \text{ as }k\rightarrow \infty.$$
Actually, we could choose $u_k$ to be a harmonic polynomial of degree $k$ with $\int_{B_3}|u_k|^2dx=1$, then it is easy to see that $\Delta u_k=0$ in $\mathbb{R}^d$ and $\int_{B_1}|u_k|^2dx=3^{-2k}\rightarrow 0$ as $k\rightarrow \infty$. Consequently, this example shows that the growth condition $(3.6)$ is necessary in the compactness argument to guarantee the Carleman-type inequality in elliptic periodic homogenization. However, we do not know that whether a Carleman-type inequality would hold without the growth condition $(3.6)$.
\end{example}
The proof of Corollary 1.4 is standard if we have obtained the Carleman-type inequality $(1.6)$. And we give it just for completeness.

\textbf{Proof of Corollary 1.4}. We just need to consider the case $0<\varepsilon\leq \varepsilon_0$, since the three-ball inequality $(1.8)$ continues to  hold for $u_\varepsilon$ if $\varepsilon\geq \varepsilon_0$ without the growth condition $(1.5)$.
According to $(1.6)$ and the choice of the cutoff function $\eta$, we have
\begin{equation}\begin{aligned}
\frac{C_0}{2} \int_{B_2\setminus B_1}\lambda^{4} \tau^{3} \varphi^{3} |u_\varepsilon|^{2} e^{2 \tau \varphi} d x
\leq \int_{B_3}\left[\mathcal{L}_{\varepsilon} (u_\varepsilon\eta+\varepsilon \chi_j^\varepsilon\partial_j\eta u_\varepsilon)\right]^{2} e^{2 \tau \varphi} d x,
\end{aligned}\end{equation}
But, in view of Lemma 3.1, there hold
\begin{equation}
\text{supp}(\mathcal{L}_{\varepsilon} (u_\varepsilon\eta+\varepsilon \chi_j^\varepsilon\partial_j\eta u_\varepsilon))\subset
 \{1/2\leq |x|\leq 2/3\}\cup \{7/3\leq |x|\leq 5/2\}
\end{equation}
and
\begin{equation}
(\mathcal{L}_{\varepsilon} (u_\varepsilon\eta+\varepsilon \chi_j^\varepsilon\partial_j\eta u_\varepsilon))^2\leq C(|u_\varepsilon|^2+|\nabla u_\varepsilon|^2).
\end{equation}Thus, we have
\begin{equation}\begin{aligned}
\int_{B_2\setminus B_1}\lambda^{4} \tau^{3} \varphi^{3} |u_\varepsilon|^{2} e^{2 \tau \varphi} d x
&\leq C\int_{\left\{B_{2/3}\setminus B_{1/2}\right\}\cup \left\{B_{5/2}\setminus B_{7/3}\right\}}(|u_\varepsilon|^2+|\nabla u_\varepsilon|^2) e^{2 \tau \varphi} d x\\
&\leq C\int_{\left\{B_{3/4}\setminus B_{1/3}\right\}\cup \left\{B_{3}\setminus B_{9/4}\right\}}(1+\lambda^2\tau^2\varphi^2)|u_\varepsilon|^{2} e^{2 \tau \varphi} d x,
\end{aligned}\end{equation} where we have used the interior Caccioppoli's inequality with weights in Lemma 2.5 in the above inequality.
Therefore, fixing $\lambda$ and changing $\tau_0$ if necessary, $(3.35)$ implies that, for $\tau_0\leq\tau_k\leq C(\lambda_0,\tau_0)\frac{||u_k||_{L^2(B_3)}}{||u_k||_{L^2(B_1)}}+100\tau_0$,
\begin{equation}\begin{aligned}
 C\int_{B_2} |u_\varepsilon|^{2} e^{2 \tau \varphi} d x
\leq (C+1)\int_{B_1}|u_\varepsilon|^2 e^{2 \tau \varphi} d x+\int_{B_{3}\setminus B_{9/4}}|u_\varepsilon|^2 e^{2 \tau \varphi} d x.
\end{aligned}\end{equation}
In view of $\varphi=\exp\left\{-\lambda|x|^2\right\}$, then it follows from that
\begin{equation}\begin{aligned}
 C\int_{B_2} |u_\varepsilon|^{2}  d x
\leq e^{\alpha\tau}(C+1)\int_{B_1}|u_\varepsilon|^2 d x+e^{-\beta\tau}\int_{B_{3}}|u_\varepsilon|^2  d x,
\end{aligned}\end{equation}where
\begin{equation}\alpha=1-2e^{-4\lambda}\quad\text{ and }\quad\beta=2\left(e^{-4\lambda}-e^{\frac{81}{16}\lambda}\right).\end{equation}
We now temporarily introduce the following notations,
$$P=(C+1)\int_{B_1}|u_\varepsilon|^2dx,\ Q=C\int_{B_2}|u_\varepsilon|^2dx\text{ and }R=\int_{B_1}|u_\varepsilon|^2dx.$$
Then $(3.37)$ becomes
\begin{equation*}
Q\leq e^{\alpha\tau}P+e^{-\beta\tau}R, \text{ for }\tau_0\leq\tau_k\leq C(\lambda_0,\tau_0)\frac{||u_k||_{L^2(B_3)}}{||u_k||_{L^2(B_1)}}+100\tau_0.
\end{equation*}
We could choose  $C(\lambda_0,\tau_0)$ such that
$$\tilde{\tau}=\frac{\ln(R/P)}{\alpha+\beta}\leq C(\lambda_0,\tau_0)\frac{||u_k||_{L^2(B_3)}}{||u_k||_{L^2(B_1)}}.$$
If $\tilde{\tau}\geq\tau_0$, then $\tau=\tilde{\tau}$ in yields that
\begin{equation}
Q\leq 2 P^{\frac{\alpha}{\alpha+\beta}}R^{\frac{\alpha}{\alpha+\beta}}.
\end{equation}
If $\tilde{\tau}<\tau_0$, $R<e^{(\alpha+\beta)\tau_0}P$ and then
\begin{equation}
Q\leq CR=CR^{\frac{\alpha}{\alpha+\beta}}R^{\frac{\alpha}{\alpha+\beta}}\leq C e^{\alpha\tau_0}P^{\frac{\alpha}{\alpha+\beta}}R^{\frac{\alpha}{\alpha+\beta}}.
\end{equation}In conclusion, we find that in any case one of inequalities $(3.39)$ and $(3.40)$ holds. That is, in terms of the original notations,
\begin{equation}
||u_\varepsilon||_{L^2(B_2)}\leq C ||u_\varepsilon||_{L^2(B_1)}^s||u_\varepsilon||_{L^2(B_3)}^{1-s},
\end{equation} with $s=\frac{\alpha}{\alpha+\beta}$.  Note that the constant $C$, in $(3.37)$, may depend on $\lambda$, then the constant $C$, in $(3.41)$, may also depend on $\lambda$. However, if $\lambda\rightarrow\infty$, the inequality $(1.8)$, with $C=M^{1/2}$, follows directly from the growth condition $(1.5)$. Therefore, we could choose the constant $C$, in $(1.8)$, that does not depend on $\lambda$.\qed\\

We are now ready to give the proof of Theorem 1.5 with the help of Corollary 1.4 and the uniform doubling conditions proved in \cite{MR3952693}.

\textbf{Proof of Theorem 1.5}.
It follows from \cite[Thm 1.2]{MR3952693} that there holds the following uniform doubling condition for $u_\varepsilon$,
\begin{equation}
\fint_{B_r}|u_\varepsilon|^2dx\leq \tilde{C}(M)\fint_{B_{r/2}}|u_\varepsilon|^2dx
\end{equation}
for any $0<r\leq 2$, with $\tilde{C}(M)$ depends only on $d$, $\tilde{M}$, $\mu$ and $M$.

It follows from Theorem 1.1 and  Corollary 1.4 that there exists $\varepsilon_0(M)$, depending only on $d$, $\tilde{M}$, $\mu$ and $M$, such that if $v\in H^2(B_4)$ solves $\mathcal{L}_\varepsilon(v)=0$ in $B_4$ and $v$ satisfies the following doubling condition
$$\fint_{B_4}|v|^2dx\leq \tilde{C}(M)\fint_{B_2}|v|^2dx,$$
then there holds the following three-sphere inequality,
\begin{equation}
||v||_{L^2(B_2)}\leq C ||v||_{L^2(B_1)}^s||v||_{L^2(B_4)}^{1-s},
\end{equation} with $s=\frac{\alpha}{\alpha+\beta}$ for $\alpha$ and $\beta$ defined in $(3.38)$

To prove the inequality $(1.9)$, we first consider the case $(\varepsilon/\varepsilon_0(M))\leq r \leq 1/2$. Let
\begin{equation}v_\varepsilon(x)=u_\varepsilon(rx),\end{equation} then it is to check that $v_\varepsilon$ satisfies
$$\operatorname{div}\left(A\left(\frac{x}{\varepsilon/r}\right)\nabla v_\varepsilon\right)=0 \text{ in } B_4,$$
with
\begin{equation*}
\fint_{B_4}|v_\varepsilon|^2dx\leq \tilde{C}(N)\fint_{B_{4}}|v_\varepsilon|^2dx,
\end{equation*} and $(\varepsilon/r)\leq \varepsilon_0(M)$, then it follows from $(3.43)$ and $v_\varepsilon(x)=u_\varepsilon(rx)$ that

\begin{equation}||u_\varepsilon||_{L^2(B_{2r})}\leq C ||u_\varepsilon||_{L^2(B_r)}^s||u_\varepsilon||_{L^2(B_{4r})}^{1-s}.\end{equation}
Suppose now that $0<r\leq(\varepsilon/\varepsilon_0(M))$, then $\frac{\varepsilon}{r}\geq \varepsilon_0(M)^{-1}$.
Therefore, it follows from the classical theory for elliptic equation with Liphschitz coefficient matrix that $(3.45)$ also holds true, with the same exponent $s$.
Thus, we have completed this proof of Theorem 1.5.

\begin{center}{\textbf{Acknowledgements}}
\end{center}

The author thanks Prof. Zhongwei Shen and Yao Xu for helpful discussions.

\normalem\bibliographystyle{plain}{}
\bibliography{carlemanbib}
\end{document}